\documentclass[a4paper]{amsart}
\usepackage{amsmath,caption,booktabs,lipsum}
\numberwithin{equation}{section}
\usepackage{amsthm}
\usepackage{amsfonts}
\usepackage{amssymb}
\usepackage{mathrsfs}
\usepackage{tikz}
\usepackage{environ}
\usepackage{elocalloc}
\usepackage{url}
\usepackage{caption}
\usepackage{graphicx}
\usepackage{CJKutf8}
\usepackage[normalem]{ulem}
\usepackage{xcolor}
\usepackage{etoolbox}
\usepackage{caption}
\usepackage{mathtools}
\usepackage{dcpic}
\usepackage{tikz-cd}
\usepackage[bottom]{footmisc}
\usepackage{hyperref}
\usepackage{enumitem}
\usepackage{stmaryrd}
\usetikzlibrary{positioning}
\usetikzlibrary{arrows,chains,positioning,scopes,quotes}
\usetikzlibrary{decorations.markings}

\newcommand\res{\mathop{\hbox{\vrule height 7pt width .3pt depth 0pt
			\vrule height .3pt width 5pt depth 0pt}}\nolimits}

\newcommand{\ai}{\alpha}
\newcommand{\be}{\beta}
\newcommand{\Ga}{\Gamma}
\newcommand{\ga}{\gamma}

\newcommand{\e}{\epsilon}
\newcommand{\hm}{\mathcal{H}}
\newcommand{\lam}{\lambda}

\newcommand{\Om}{\Omega}
\newcommand{\fr}{\textnormal{FocalRad}}

\newcommand{\Si}{\Sigma}

\newcommand{\rh}{\rho}
\newcommand{\ta}{\theta}
\newcommand{\Ta}{\Theta}

\newcommand{\cms}{\operatorname{comass}}

\newcommand{\ms}{\mathbf{M}}

\newcommand{\cd}{\cdots}
\newcommand{\T}{\mathbf{T}}
\newcommand{\s}{\subset}

\newcommand{\cp}{^\complement}

\newcommand{\la}{\langle}
\newcommand{\ra}{\rangle}

\newcommand{\no}[1]{\left\lVert#1\right\rVert}
\newcommand{\nog}[1]{\left\lVert#1\right\rVert_g}

\DeclarePairedDelimiter{\ri}{\la}{\ra}

\newcommand{\spt}{\textnormal{spt}}

\newcommand{\m}{^{-1}}
\newcommand{\ts}{\otimes}

\newcommand{\pd}{\partial}

\newcommand{\na}{\nabla}

\newcommand{\R}{\mathbb{R}}
\newcommand{\Z}{\mathbb{Z}}

\makeatother
\theoremstyle{plain}
\newtheorem{thm}{Theorem}[section]
\newtheorem{lem}[thm]{Lemma}

\newtheorem{cor}{Corollary}

\theoremstyle{definition}
\newtheorem{defn}{Definition}[section]

\theoremstyle{remark}
\newtheorem{rem}{Remark}

\title[Asymptotic regularity of mod $v$ area-minimizers]{Homologically area-minimizing surfaces mod $v$ have at worst codimension 2 singular sets asymptotically}
\author{Zhenhua Liu}
\dedicatory{Dedicated to Xunjing Wei}
\begin{document}
	\maketitle\vspace{-3em}
	\begin{abstract}
		In \cite{DHMS}, De Lellis and coauthors have proved a sharp regularity theorem for area-minimizing currents in finite coefficient homology. They prove that area-minimizing mod $v$ currents are smooth outside of a singular set of codimension at least $1.$ Classical examples like triple junctions demonstrate that their result is sharp.
		
		Surprisingly, even though their regularity theorem cannot be improved for any fixed $v$, if one instead fixes the homology class, then $v$ asymptotically one can always achieve more regularity.
		For any integral homology class $[\Si]$ on any Riemannian manifold, we show that for $v$ large, any area-minimizing mod $v$ current in $[\Si\mod v]$ must be an integral current, thus having a singular set of codimension at least $2$ in general and of codimension at least $7$ in the hypersurface case. Similar results are obtained for Plateau problems in Euclidean space. Our work is inspired by Morgan's work \cite{FM4} and based on De Lellis' and coauthors' work \cite{DHMS}. 
	\end{abstract}
	\tableofcontents
	\section{Introduction}
	In this manuscript, area-minimizing surfaces refer to area-minimizing integral currents and mod $v$ currents for integers $v\ge 2$, which roughly speaking are submanifolds counted with multiplicity in $\Z$ or $\Z/v\Z$ coefficient, respectively, minimizing the area functional with homologous competitors. 
	
	For $\Z$-coefficient homology,  Almgren's Big Regularity Theorem (\cite{A}) and De Lellis-Spadaro-Minter-Skorobogatova (\cite{DS1}\cite{DS2}\cite{DS3}\cite{DMS}) show that $n$-dimensional area minimizing integral currents are smooth manifolds outside of a rectifiable singular set of dimension at most $n-2.$ (In the codimension 1 case, the dimension of the singular set can be reduced to $n-7$ by \cite{HF1}\cite{NV}.) 
	
	For $\Z/v\Z$-coefficient homology, De Lellis-Hirsch-Marchese-Stuvard (\cite{DHMS}) show that $n$-dimensional area minimizing integral currents are smooth manifolds outside of a singular set of dimension at most $n-1$. (In the hypersurface case, the singular set consists of transversely intersecting hypersurfaces along a common boundary outside of a set of codimension at least $2$ \cite{DHMSS}\cite{DHMSS2}\cite{MW}. This refined regularity theory however will not be needed in this manuscript. In the hypersurface case \cite{BW1} suffices for our purposes.) 
	
	Concrete examples like holomorphic subvarieties and triple junctions show that both the $\Z$ and the $\Z/v\Z$ regularity theorems above are sharp, respectively. Thus, it seems impossible that the dimension difference in regularity theorems between the $\Z$ and $\Z/v\Z$ coefficient can be bridged in general.
	
	However, pioneering work by Frank Morgan (\cite{FM4}) has shown that for boundaries with at most $v/2$ components on convex sets,  mod $v$ hypersurface solutions to Plateau problems must be integral, thus having a singular set of codimension at least $7$. To the author's knowledge, Morgan's theorem (\cite{FM4}) remains the only known instance in the literature about such unexpected improvement of regularity and is largely regarded as a special case by experts. To quote verbatim Minter-Wickramasekera (\cite{MW}),\\
	"\textbf{although
		in general such singularities in T obviously do arise, they can in certain special
		circumstances be ruled out; for instance, by a theorem of F. Morgan} \cite{FM4}"
	
	However, inspired by \cite{FM4} and in sharp contrast to the literature, we show that asymptotically we can always annihilate the difference in regularity between $\Z$ and $\Z/v\Z$ coefficient area-minimizers, in the setting of both homological area-minimizers and Plateau problems.
	\begin{thm}
		Let $[\Si]$ be a $d$-dimensional integral homology class on a $d+c$-dimensional smooth compact closed Riemannian manifold $M^{d+c}$, with $d,c\ge1.$ Then for $v$ large enough, any area-minimizing mod $v$ current in $[\Si\mod v]$ is smooth outside of a singular set of codimension at least $2$ for $c\ge 2$ and codimension at least $7$ for $c\ge 1.$
	\end{thm}
	The above theorem is a direct corollary of the following much stronger result.
	\begin{thm}\label{modvc}
		Assume
		\begin{itemize}
			\item
			$d,c\ge 1,v\ge 2$ are integers,
			\item $M^{d+c}$ is a compact closed (not necessarily orientable) smooth manifold,\item $S$ is a finite collection of $d$-dimensional integral homology classes on $M$,
			\item $K$ is a compact subset of the space of Riemannian metrics on $M.$
		\end{itemize}
		Then  there exists an open set $\Om_{K,S}$ in the space of Riemannian metrics with $K\s \Om_{K,S}$, and an integer $\nu_{K,S},$ such that,\begin{itemize}
			\item for any metric $g\in\Om_{K,S}$, any $v$ with $v\ge \nu_{K,S}$, any mod $v$ area-minimizing current $T$ in any homology class $[\Si\mod v] $ with $[\Si]\in S$ is an integral current, thus smooth outside of a singular set of dimension at most $d-2$ for $c\ge 2$  and at most $d-7$ for $c=1$,
		\end{itemize}
	\end{thm}
	\begin{rem}
		Compact subsets in the space of Riemannian metrics come easily by taking finite-dimensional smooth families of Riemannian metrics. 
	\end{rem}
	\begin{rem}
		Note that the assumption of an integral $[\Si]$ before taking mod $v$ cannot be dropped. It is equivalent to saying that $[\Si\mod v]$ has integral representatives. By the universal coefficient theorem, there are indeed homology classes that can never be represented by integral currents if $H_{d-1}(M,\Z)$ has suitable nontrivial torsion, e.g., $M=\mathbb{RP}^{d+c}$.
	\end{rem}
	\begin{rem}
		We do not claim that $T\in [\Si]$. We only know that $T\in[\Si]+v[\Pi]$ for some integral homology class $[\Pi].$ When $H_d(M,\Z)$ is torsion-free, then it is indeed true that for $v$ large enough $T\in [\Si],$ with a proof similar to that of Lemma \ref{contm}. However, for $H_d(M,\Z)$ with torsion subgroups, concrete examples show that one can have $T\in[\Si]+v[\Pi]$ for $v\to\infty$ and possibly different $[\Pi].$
	\end{rem}
	A much stronger conclusion is true if we consider the set of all integral homology classes mod $v.$
	\begin{thm}\label{modva}
		With the same assumption as in Theorem \ref{modvc}, in any smooth metric on $M$, we have
		\begin{align*}
			\liminf_{v\to\infty}\frac{\#\{ H_d(M,\Z/v\Z)\ni[\Pi]\textnormal{ having only integral area-minimizers}\}}{\#\{ H_d(M,\Z/v\Z)\ni[\Pi]\textnormal{ admitting integral representatives}\}}>0.
		\end{align*}
		Here $\#$ means the number of elements in the set.
	\end{thm}
	\begin{rem}
		By the universal coefficient theorem, the denominator above equals
		\begin{align*}
			\# H_d(M,\Z)\ts \Z/v\Z.
		\end{align*}
		Consequently $\# H_d(M,\Z)\ts \Z/v\Z\ge v^{b_d}\to\infty$ as $v\to\infty$, where $b_d$ is the $d$-th Betti number of $M$.
	\end{rem}
	For Plateau problems with boundaries, similar conclusions hold in Euclidean space.
	\begin{thm}\label{modvb} Suppose
		\begin{itemize}
			\item $d\ge 1,c\ge 1$ are integers,
			\item $\Ga=\sum_j t_j\Ga_j$ is a smooth integral current, with $t_j\in\Z,$ $\Ga_j\s \R^{d+c}$ a finite collection of pairwise disjoint connected orientable $d$-dimensional compact submanifolds of $\R^{d+c}.$
		\end{itemize}
		Then there is $\nu_\Ga>0,$ such that for $v\ge \nu_\Ga$, any area-minimizing mod $v$ current $T$ with $$\pd T=\Ga\mod v$$ is an integral current with boundary $\Ga$, thus smooth in the interior outside of a singular set of dimension at most $d-2$ for $c\ge 2$  and at most $d-7$ for $c=1$.
	\end{thm}
	\begin{rem}
		Theorem \ref{modva} and Theorem \ref{modvb} can be strengthened in the style of Theorem \ref{modvc} about smooth families of boundaries and smooth families of metrics.
	\end{rem}
	The above phenomenon is deeply connected to the Lavrentiev gap of the minimal area between integral and mod $v$ coefficient homology. To discuss these, we need some definitions.
	\begin{defn}\label{minm}
		With the same assumption as in Theorem \ref{modvc}, let $[\Si]$ be a $d$-dimensional integral homology class, define
		\begin{align*}
			\inf_\Z\ms_g ([\Si])&=\inf_{Z\in [\Si]}\ms_g(Z),\\
			\inf_{\Z}\ms_g ([\Si\mod v])&=\inf_{[\Pi]\in H_d(M,\Z)}\inf_\Z \ms_g( [\Si]+v[\Pi]),\\
			\inf_{\Z/v\Z}\ms_g([\Si\mod v])&=\inf_{Z\in[\Si\mod v]\in H_d(M,\Z/v\Z)}\ms_g^v(Z),\\
			\nog{[\Si]}&=\inf_{Z\in [\Si]\in H_d(M,\R)}\ms_g(Z).
		\end{align*}
	\end{defn}
	Here the subscript $g$ is the Riemannian metric and the superscript $v$ indicates mass mod $v$. The inclusion into $\Z,\Z/v\Z,\R$ coefficient homology means the natural map by universal coefficient theorem and the objects inside the homology classes are integral currents, mod $v$ currents, and real flat chains, respectively.
	
	To simplify notations, we will write $\nog{\Si}$ instead of $\nog{[\Si]}.$ Also, when there is no confusion, we will omit the subscript $g.$
	\begin{rem}
		The reader might wonder why we use $\nog{\cdot}$ instead of $\inf_\R\ms_g([\Si]).$ Federer (\cite{HF2}) proved that  $\nog{\cdot}$ is a well-defined continuous norm on the real homology groups of $M$ and this is a crucial fact we use many times. Details can be found in Section \ref{nm}.
	\end{rem}		
	First of all, by definition we have
	\begin{align}\label{lavi}
		\inf_\Z\ms_g([\Si\mod v])\ge
		\inf_{\Z/v\Z}\ms_g([\Si\mod v]).
	\end{align}
	\begin{defn}
		Define
		\begin{align*}
			K_{[\Si\mod v]}
		\end{align*} to be the set of metrics $g$ on $M$, so that $\inf_\Z\ms_g([\Si\mod v])=
		\inf_{\Z/v\Z}\ms_g([\Si\mod v]).$
	\end{defn}
	\begin{rem}
		When $d=1,v=2,$ by regularity of $1$-d minimizing currents (\cite{HF1}), equality always holds in (\ref{lavi}).\end{rem}
	We have
	\begin{cor}\label{setk}
		\begin{enumerate}
			\item $			K_{[\Si\mod v]}$ is always a (possibly empty) closed set,
			and	equality in (\ref{lavi}) holds if and only if some area-minimizing current mod $v$ in $[\Si\mod v]$ is integral.
			\item There exists $\nu>0,$ so that
			\begin{align*}
				\bigcap_{v\ge \nu}K_{[\Si\mod v]}
			\end{align*}is a closed set with a non-empty interior.
			\item \begin{align*}
				\bigcup_vK_{[\Si\mod v]}
			\end{align*}equals the set of all smooth Riemannian metrics on $M.$
		\end{enumerate}
	\end{cor}
	We show that Theorem \ref{modvc} is sharp in the sense that $K_{[\Si\mod v]}\cp$ is non-empty in many cases as follows.
	\begin{thm}\label{lav}
		With the same assumption as in Theorem \ref{modvc}, suppose
		\begin{itemize}
			\item $c\ge d\ge 2,$ or $c=1,d\ge 2,$ $v$ is an integer with $v\equiv 2(\mod 4)$, and $R$ is a positive real number,
			\item $[\Si]$ is an integral homology class so that $2[\Si\mod v]=0.$
		\end{itemize}
		Then \begin{itemize}
			\item there exists a smooth Riemannian metric $g$ on $M$ so that
			\begin{align*}
				\frac{\inf_\Z\ms_g([\Si\mod v])
				}{\inf_{\Z/v\Z}\ms_g([\Si\mod v])}\ge R
			\end{align*}
			In other words, the minimum area among mod $v$ currents is arbitrarily small compared to integral currents, and $K_{[\Si\mod v]}\cp$ is a non-empty open set.
		\end{itemize}
	\end{thm}
	\begin{rem}
		Note that if the $d$-th Betti number of $M$ is non-zero, then such $[\Si]$ always exists.
	\end{rem}
	\begin{rem}
		Frank Morgan has communicated to the author the hypersurface case of the above theorem, and the credit goes to him for the proof of $c=1,d\ge 2$. Furthermore, Morgan has pointed out that on compact manifolds with $\Z/v\Z$ symmetry, one can construct mod $v$ area-minimizers that have mod $v$ junction-like singularities, i.e., non-empty codimension 1 strata. Thus, Theorem \ref{modvc} is indeed sharp.
	\end{rem}
	Similar result holds for Theorem \ref{modvb}.
	\begin{thm}\label{lavb}
		Let $d\ge 2,c\ge 1$ be integers. Then for any $v\ge 2,$ $R>0,$ there exists a smooth $d-1$-dimensional connected embedded compact oriented submanifold $\Ga$ of $\R^{d+c}$, so that
		\begin{align*}
			\frac{\inf\{\ms(T)|\pd T=\Ga\}}{\inf\{\ms^v(T)|\pd T=\Ga\mod v\}}\ge R.
		\end{align*}
	\end{thm}
	The above results can be compared with Federer's work \cite{HF2} and the author's work \cite{ZLc} on the Lavrentiev gap between $\Z$ coefficient and $\R$ coefficient minimal mass, which is more prevalent. Also, compared to Theorem \ref{modvc}, it is not expected that real coefficient mass minimizers in general lie in integral homology.
	\subsection{Sketch of proof}Consider any mod $v$ area-minimizer in $[\Si].$ Take any of its representing lifts to rectifiable currents, denoted by $T$.
	The key observation is that the monotonicity formula gives a priori density upper bounds on mod $v$ minimizers, forcing every point to have a density smaller than $v/2$ for $v$ large. This strict density bound has three consequences: 
	\begin{enumerate}
		\item the boundary of $T$ is supported in the singular set, 
		\item the codimension $1$ strata of the singular set is empty, 
		\item as per Theorem 1.6 of \cite{DHM}, the top strata singular points have codimension at least two. 
	\end{enumerate}These observations, combined with Almgren's stratification of singular sets, imply that the boundary of the rectifiable lift, $T$, has at least codimension $2$.   By classical properties of integral flat chains, this means $T$ must be an integral cycle.
	
	A more quantitative approach to these arguments proves Theorem \ref{modvc} and \ref{modva}.  Corollary \ref{setk} follows directly from these results. For the Plateau problems in Euclidean space (Theorem \ref{modvb}), the method is analogous. Theorem \ref{lav} leverages Frank Morgan's calibration modulo $v$ (\cite{FM4}) and Yongsheng Zhang's constructions (\cite{YZa}\cite{YZj}), while Theorem \ref{lavb} is based on a simple perturbation of $v$ multiples of a $d-1$ dimensional unit sphere.
	\section*{Acknowledgements}
	I cannot thank my advisor Professor Camillo De Lellis enough for his unwavering support while I have been recovering from illness. I feel so lucky that I have Camillo as my advisor. Many thanks go to him for countless helpful suggestions regarding this manuscript and others. The author also wants to thank Professor Frank Morgan for his constant support and pioneering work that has inspired many constructions in the author's works. Many thanks go to him for countless helpful conversations. A special thank goes to Professor Bruno Martelli, whose Mathoverflow answer \cite{BM} is essential to the $c=1$ case of Theorem \ref{lav}. \section{Preliminaries and Lemmas}
	In this section, we will fix our notations and prove some basic lemmas.
	\subsection{Manifolds and neighborhoods}\label{mfd}
	We will reserve $M$ to denote an ambient smooth compact closed Riemannian manifold. Submanifolds will be denoted by $N,L,Q$, etc. We will use the following sets of definitions and notations.
	\newcommand{\inj}{\textnormal{Inj}}
	\begin{defn}\label{defns}
		Let $N$ be a smooth submanifold of a Riemannian manifold $(M,g).$
		\begin{itemize}
			\item $B_r^{(M,g)}(N)$ denotes a tubular neighborhood of $N$ inside $M$ in the intrinsic metric on $M$ of radius $r$. When there is no a priori choice of metric, then use an arbitrary metric.
			\item $\inj_g(M)$ denotes the injectivity radius of $M$ in metric $g.$
			\item $\T_p M$ denotes the tangent space to $M$ at $p.$ We will often regard $\T_p N$ as a subspace of $\T_p M.$
			\item $\exp^\perp_{N\s M}$ denotes the normal bundle exponential map of the inclusion $N\s M.$ 
			\item $\fr_N^M$ is the focal radius of $N$ in $M,$ i.e., the radius below which the normal bundle exponential map remains injective.
			\item $\pi_N^M$ denotes the nearest distance/normal bundle exponential map projection from $M$ to $N$ in the metric intrinsic of $M$ in $B_r^N(M)$ with $r$ less than the focal radius of $N$ inside $M$.
		\end{itemize}   
	\end{defn}
	\begin{rem}
		We will often drop the sup/subscripts $M,N$ when there is no confusion.
	\end{rem}
	\subsection{Algebraic topology}\label{at}The main reference is \cite{AH}. For a compact manifold $M,$
	$H_d(M,\Z)$ is a finitely generated abelian group. We say a class is torsion if some multiple of it is zero, and call a class non-torsion otherwise.
	
	We fix a decomposition of the homology group as follows:
	\begin{itemize}
		\item  ${[\Si_1],\cd,[\Si_b],[\tau_1],\cd,[\tau_a]}$ is a $\Z$-linear independent generating set of $H_d(M,\Z)$ over $\Z,$
		\item all $[\Si_j]$ are non-torsion and all $[\tau_j]$ are torsion,
		\item $N=\max_j\deg[\tau_j],$ where the degree is the least positive integer which multiplies the torsion class to zero.
	\end{itemize}
	The reader should be familiar with the universal coefficient theorem, which will remain fundamental to all of our discussions. The natural map from integral to mod $v$ homology will be denoted by $[\Si]\mapsto[\Si\mod v].$\subsection{Integral currents and flat chains}\label{cfc}
	In this manuscript, we need several notions of currents. For a comprehensive introduction, the standard references are \cite{LS1} and \cite{HF}.  For our purposes, the reader suffices to know the following from 4.1.24 in \cite{HF}
	\begin{itemize}
		\item $d$-dimensional rectifiable currents are finite mass limits of the Lipschitz image of oriented polyhedron chains under flat topology.
		\item $d$-dimensional integral currents are $d$-dimensional rectifiable currents with boundary being $d-1$-dimensional rectifiable currents.
		\item $d$-dimensional integral flat chains decompose as sum of $d$-dimensional rectifiable currents and boundary of $d+1$-dimensional rectifiable currents. 
	\end{itemize}
	By 4.2.16 of \cite{HF}, integral currents are rectifiable currents with finite boundary mass. Rectifiable currents are integral flat chains with finite mass.
	We will use the following definition of the irreducibility of currents.
	\begin{defn}\label{irr}
		A closed integral current $T$ is irreducible in $U,$ if we cannot write $T=S+W,$ with $\pd S=\pd W=0,$ and $S,W$ nonzero, and $\ms(T)=\ms(S)+\ms(W).$
	\end{defn}
	\begin{rem}
		It is easy to see that the definition of irreducibility above is equivalent to 1.c.i of Section 2.7 of \cite{ZL1}, if we assume the other assumptions in that Section.
	\end{rem}
	Often we will abuse the notations by using $T$ to denote $\spt T$. Also, we will use the differential geometry convention of closedness. An integral current $T$ is closed if $\pd T=0$ and $T$ has compact support. 
	\subsection{Mod $v$ currents}
	By mod $v$ currents, we mean flat chains modulo $v$. The main reference is Section 4.2.26 of \cite{HF1}, \cite{RY} and \cite{DHMS}. Roughly speaking, they are the closure of polyhedron chains mod $v$ under mod $v$ flat topology. For our purposes, it suffices to know that they form a complete set under mod $v$ flat topology (page 424 of \cite{HF1}), and the chain complex formed by them induces precisely the homology with mod $v$ coefficient  (4.4.5 of \cite{HF1}). 
	
	There is a natural map from integral flat chains to flat chains mod $v$, i.e., taking its mod $v$ equivalence class (\cite{RY}, 4.2.26 of \cite{HF}). We will use $T\mapsto T\mod v$ to denote this map.
	
	Conversely, every mod $v$ current $T$ of finite mass has at least one representative modulo $v$ (page 430 of \cite{HF}), i.e., a rectifiable current $S$ with the same induced mod $v$ mass measure and in the same mod $v$ equivalence class as $T.$
	
	Throughout the manuscript, we will speak of mod $v$ currents that are integral currents. The definition is as follows.
	\begin{defn}Suppose
		\begin{itemize}
			\item $T$ is a mod $v$ current, $\Ga$ is an integral current,
			\item $\pd T=\Ga\mod v.$
		\end{itemize}	We say $T$ is an integral current if
		\begin{itemize}
			\item at least one representative modulo $v$, say $S$, of $T$ is an integral current,
			\item $\spt\pd S\s\spt\Ga.$
		\end{itemize}
	\end{defn}
	\subsection{Mass and comass}
	The comass of a $d$-dimensional differential form $\phi$ in a metric $g$ is defined as
	\begin{align*}
		\cms_g \phi=\sup_x\sup_{P\s\T_xM}\phi(\frac{P}{|P|_g}),
	\end{align*}where $P$ ranges over $d$-dimensional oriented planes in the tangent space to $M.$ 
	
	The mass of an integral current is defined as
	\begin{align*}
		\ms_g(T)=\sup_{\cms_g(\phi)\le 1}T(\phi).
	\end{align*}
	
	The definition of mass mod $v$ is more complicated, but roughly
	$$	\ms^v(T)
	$$ equals the mass of the least mass of rectifiable currents in the same modulo class (page 430 of \cite{HF}).
	\subsection{Representing homology classes by submanifolds}
	In this section, we will collect several facts about representing homology classes using submanifolds.
	\begin{lem}\label{nono}
		For $d\ge 2,c\ge 2,$ there exists a connected embedded $d$-dimensional non-orientable submanifold $L^d$ in $\R^{d+c}.$
	\end{lem}
	\begin{proof}
		It is well known that $\mathbb{RP}^2$ embeds into the unit ball in $\R^4,$ e.g., \cite{GL}. Consider the following map $
		u:\R^4\times \R^{d-1}\to \R^{d+2}$ defined by
		\begin{align*}
			u(x_1,\cd,x_4,y_1,\cd,y_{d-1})=((4+x_1)y_1,\cd,(4+x_{1})y_{d-1},x_2,x_3,x_4).	
		\end{align*}
		Straightforward calculation shows that $u$ is an embedding restricted to $B_2^4(0)\times S^{d-2}.$ Thus, $u(\mathbb{RP}^2\times S^{d-2})$ is an embedding of $\mathbb{RP}^2\times S^{d-2}$ into $\R^{d+2}$. Note that $\mathbb{RP}^2\times S^{d-2}$ is not orientable. (Suppose not. Equip the product with a product metric. Interior product of the volume form on the sphere factor with the volume form on the product gives a volume form on $\mathbb{RP}^2,$ a contradiction.) Now consider any standard embedding $\R^{d+2}\s\R^{d+c}$ and compose it with $u.$ We are done.
	\end{proof}
	\begin{lem}\label{thom}
		With the assumptions in Theorem \ref{lavb}, there exists a closed connected non-orientable smooth submanifold $Q$ of $M,$ so that 
		\begin{align*}
			[\Si\mod v]=[\frac{v}{2}Q\mod v].
		\end{align*}
	\end{lem}
	\begin{proof}
		Let us deal with the case of $c\ge d\ge 2$ first. As mentioned in the second paragraph of \cite{BD}, \cite{CW1} and \cite{RT} show that for any integral homology class $[\Si]$ with $d,c\ge 1$, there is an odd integer $\lam$ so that $\lam[\Si]$ can be represented by continuous maps from smooth manifolds. See also \cite{VB} for sharp bounds on $\lam$. By Whitney's approximation theorem (Theorem 6.26 of \cite{JL}) and the denseness of embeddings in the mapping space (Theorem 4.7.7 of \cite{CW}) when $c>d,$ we can get a smoothly embedded submanifold $N$ representing $\lam[\Si].$ When $c=d,$ the above argument gives an immersed representative of $\lam[\Si]$ with transverse double points only. However, one can always replace the double points with necks to get a smoothly embedded representative of $\lam[\Si]$ as well.		
		
		To sum it up, for some odd number $\lam$, $\lam[\Si]$ can be represented by an embedded orientable smooth submanifold $N$. Since the codimension is larger than $1$, using connected sums, one can assume that $N$ is connected (detailed argument in the proof of Lemma 2.1 in \cite{ZLc}).  Since $\lam,\frac{v}{2}$ is odd and $[2\Si\mod v]=0$, we deduce that
		\begin{align*}
			[\frac{v}{2}N\mod v]=[\frac{v}{2}\lam\Si\mod v]=[\Si\mod v].
		\end{align*}
		Pick a point $p\in N,$ and a small coordinate  $(x_1,\cd,x_{d+c})$ ball $B_r(p)$ around $p$ so that $N$ becomes the $x_1,\cd,x_d$-plane. By Lemma \ref{nono}, there exists a non-orientable manifold $L^d$ that can be embedded into $B_r(p)\cap\{x_{d+c}>0\},$ since the latter is diffeomorphic to $d+c$-dimensional standard ball. Pick a point $q$ on $L$ and take a curve $\ga$ from $p$ to $q$ that is disjoint from $N,L$ except at the endpoints. Make a tubular neighborhood of $\ga$ into a neck and make a connected sum of $N\#L$. Note that $N\# L$ is not orientable, otherwise, $L$ minus a small disk is orientable. Then adding the disk will keep orientability, a contradiction.
		
		Set $Q=N\#L$. Then we claim that $[\frac{v}{2}Q\mod v]=[\frac{v}{2}N\mod v]=[\Si\mod v]$.
		
		To see this, let us show that $[\frac{v}{2}L\mod v]$ is a boundary. Note that  $\frac{v}{2}L$ is indeed a mod $v$ cycle. Consider $\mathbb{RP}^2$ as an oriented chain with boundary $2\mathbb{RP}^1$ and take the chain structure to $L=\mathbb{RP}^2\times S^{d-2}.$ Since the $d$-dimensional mod $v$ homology of $B_r(p)$ is trivial by homotopy invariance, we deduce that $\frac{v}{2}L$ is indeed a mod $v$ boundary.
		
		The necks we use to do connected sums are tubular neighborhoods of $\ga.$ By taking a continuous family of ambient diffeomorphisms that shrink a neighborhood around $\ga$ to $\ga$, we deduce that $[\frac{v}{2}Q]=[\frac{v}{2}N]+[\frac{v}{2}L]$ by homotopy invariance (\cite{AH}). We are done with the case $c\ge d\ge 2$.
		
		For the case of $c=1,d\ge 2,$ by Corollary 3.28 of \cite{AH}, $H_d(M,\Z)\equiv \Z^{b_d},$ with $b_d$ the $d$-th Betti number of $M.$ It is straightforward to check that our assumptions imply that $[\Si\mod v]=\frac{v}{2}[\Pi\mod v]$ for some primitive homology class $\Pi\in H_d(M,\Z).$ By point 2 of Representation Theorem of \cite{WM}, $[\Pi]$ can be represented by a connected closed smooth hypersurface $N$ and $M\setminus N$ has one path component with two ends. Instead of doing connected sums with $L,$ which does not exist for $c=1,$ we do a self-connected sum to destroy the orientation of $N,$ as pointed out by \cite{BM}. Since $N\setminus M$ has only one path component, take a smooth curve in $M\setminus N$ to connect its two ends. Do a connected sum of $N$ along $\ga$ with itself. Call the result $Q.$ It is straightforward to check that $Q$ is unorientable and $Q$ is homologous to $N$ as mod $2$ chains. Now apply the same argument as in the case of $c\ge d\ge 2,$ we deduce that $[\Si\mod v]=[\frac{v}{2}Q\mod v].$ We are done.  \end{proof}
	\begin{lem}\label{retr}
		With the same assumptions as in Theorem \ref{modvc}, suppose
		\begin{itemize}
			\item $[\Pi]\in H_d(M,\Z/v\Z)$ is a mod $v$ homology class with $v\ge 2$,
			\item there is a smooth connected not necessarily orientable manifold $Q$  and an integer $0<t\le \frac{v}{2},$
			\item $[\Pi]=[tQ].$
		\end{itemize}
		Then for any $R>1,$	 there exists a smooth metric $g$, and a neighborhood $U$ around $Q,$ such that\begin{itemize}
			\item $U$ deformation retracts onto $Q$,
			\item $tQ$ is homologically mod $v$ area-minimizing in $(M,g)$,
			\item any stationary varifold that is not compactly contained in $U$ must have mass at least $Rt\ms(Q)$.
		\end{itemize}
	\end{lem}
	\begin{proof}
		First, equip $M$ with an arbitrary smooth metric $h.$ Then take $U=B_{\frac{r}{3}}(Q)$ with any chosen $r<\textnormal{FocalRad}_Q^M.$ By construction, $U$ deformation retracts onto $Q.$
		
		Even if $Q$ is not orientable, $\no{(\pi_Q^M)dvol_Q}$ is a well-defined smooth function in $B_{r}(Q),$ where $dvol_Q$ is any locally chosen volume form of $Q.$
		
		By Remark 3.5 of \cite{YZa} and Lemma 5.1 of \cite{ZL1}, we deduce that $\pi_Q^M$ is an area-non-increasing projection in the metric $\no{(\pi_Q^M)dvol_Q}^{\frac{2}{d}}h.$
		
		For any $R> 1,$ using the proof of Theorem 4.1 in \cite{YZj}, there exists a smooth function $f$, so that $f$ equals $1$ near $Q,$ and in $fh,$ $\pi_Q^{(M,h)}$ is still area-non-increasing in $B_r^h(Q).$ Moreover, any stationary varifold in $(M,fh)$ that is not compactly contained in $B_{\frac{r}{3}}^h(Q)$ has area at least $R\ms_h(tQ)$.
		
		We claim that $tQ$ is mod $v$ area-minimizing in $(M,fh)$. To see this, suppose $T$ is an area-minimizing current mod $v$ in $[\Si]=[tQ].$ Then by construction, $T$ must be contained in $B_{\frac{r}{3}}^h(Q),$ otherwise having more area than $tQ.$ However, if $T$ is contained in $B_{\frac{r}{3}}(Q),$ then 
		\begin{align*}
			\ms^v(\pi_Q^{(M,h)}(T))\le\ms^v(T),
		\end{align*}thus $\pi_Q^{(M,h)}(T)$ must also be area-minimizing mod $v.$
		Since $\pi_N^M$ is a deformation retract, we deduce that $\pi_Q^M(T)$ is supported in $Q$ and quals to $tQ\mod v.$ Since $1\le t\le \frac{v}{2},$ we must have $\ms^v(\pi_Q^M(T))=\ms(tQ).$ We are done.
	\end{proof}
	\subsection{A priori density bounds}
	In this subsection, we will collect several a priori density bounds about mod $v$ currents.
	\begin{lem}\label{mon}
		With the same assumptions as in Theorem \ref{modvc}, let $g$ be a smooth Riemannian metric on $M.$
		Then there exist an open set of smooth Riemannian metrics $\Om_g$ which contains $g$, and constants $c_g,r_g>0$,  such that,	\begin{itemize}
			\item for any $d$-dimensional stationary integral varifold $V$ on $M$ in metric $g'\in \Om_g$, any point $p\in M$ and  radii $r\in(0,r_g],$ the function
			\begin{align*}
				\exp(c_gr)V(B_r^{g'}(p))r^{-d}
			\end{align*} is monotonically increasing in $r$, where $V(B_r^{g'}(p))$ is the measure $V$ of the radius $r$ ball centered at $p$ in $g'.$
			\item the density $\ta_p$ of $V$ stationary in any $g'\in \Om_g,$ at any point $p\in M$ is well-defined and there is a constant $C>0$, so that 
			\begin{align*}
				\ta_p< C V(M).
			\end{align*}
		\end{itemize}
	\end{lem}
	\begin{proof}
		This is folklore and we will only give a sketch of the proof. First consider everything in the metric $g,$ for any point $p\in M$ and $\rh=\frac{1}{2}\inj_g(M),$ adopt a normal coordinate $(x_1,\cd,x_{d+c})$ in $B_\rh(p)$. Let $r$ denote the distance to $p$ then by Chapter 2 in \cite{AG} we have 
		\begin{align*}
			r\na r=\sum_j x_j\pd_j.
		\end{align*}
		A straightforward calculation using Taylor expansion shows that 
		\begin{align*}
			|\ri{\na e(r\na r),e}-1|\le c,
		\end{align*}
		where $e$ is any unit length tangent vector in $B_\rh(p).$ Here the constant $c$ depends is controlled by the second derivatives of the metric.
		
		Arguing as in the first two sections of \cite{D}, we deduce that $\exp(cr)V(B_r^g(p))r^{-d}$ is monotonically increasing in $r$.
		
		Now if one varies the point $p$, since $c$ is controlled by $\no{g}_{C^3},$ by the compactness of $M,$ we deduce the existence of a uniform $c$ that works for all $p.$ Note that all of the above is independent of $V.$
		
		By Section 8 Theorem in \cite{PE}, $\inj_g(M)$ depends continuously on the metric. Thus, there exists an open neighborhood $\Om_g$ in the space of Riemannian metrics so that $g\in\Om_g$ and $r_g=\frac{1}{2}\inf_{g'\in \Om_g}\inj_{g'}(M)>0.$ However for each $g'\in \Om_g$, running the above argument gives a monotonicity formula with a constant $c$ that is controlled by $C^3$ norm of the metric. Thus, by shrinking $\Om_g$ if necessary, we can get a uniform upper bound $c_r$. Since the product of non-negative monotonically increasing functions preserves the monotonicity, we are done with the first two claims. 
		
		The last claim follows from the first two points and \begin{align}
			\ta_p=\lim_{r\to 0}\frac{V(B_r^{g'}(p))}{r^d}\le \exp(c_gr_g)V(B_{r_g}^{g'}(p))\le CV(M).
		\end{align}   
	\end{proof}
	\begin{lem}\label{dst}
		With the same assumptions as in Theorem \ref{modvc}, there exists an open set $\Om_{K,S}$ in the space of Riemannian metrics with $K\s \Om_{K,S}$, and a real number $\ta_{K,S}$ such that
		\begin{itemize}
			\item for any metric $g\in\Om_{K,S}$, any $v$ with $v\ge 2$, any area-minimizing mod $v$ current $T$  in any homology class $[\Si\mod v] $ with $[\Si]\in S$ has density at most $\ta_{K,S}$ at any point.
		\end{itemize}
	\end{lem}
	\begin{proof}
		First, it suffices to deal with the case $S=\{[\Si]\},$ $K=\{g\}$, i.e., both having only one element. To see this, suppose we have already proved the lemma for the one-element-only case. Then by finiteness of $S$, setting
		$\Om_{g,S}=\cap_{[\Si]\in S}\Om_{g,[\Si]}$ and $\ta_{g,S}=\max_{[\Si]\in S}\ta_{g,S}$ proves the case for $K=g$ and $S$ arbitrary. Now, $\cup_{g\in K}\Om_{g,S}$ forms an open cover of $K,$ so we can choose a finite subcover $K\s \cup_{j=1}^k\Om_{g_j,S}.$ Finally, setting $\Om_{K,S}=\cup_{j=1}^n\Om_{g_j,S}$ and $\ta_{K,S}=\max_{j=1}^n\ta_{g_j,S}$ and we are done.
		
		With the reduction in hand, it suffices to prove the one-element-only case.
		
		Apply Lemma \ref{mon} to deduce that there exists an open set of Riemannian metrics $\Om_g\supset g,$ so that for any stationary varifold $V$ in any metric $g'\in \Om_g$, for any point $p,$ we have
		\begin{align}\label{dbd}
			\ta_p< CV(M).
		\end{align}
		Let $T$ be an area-minimizing mod $v$ current in $[\Si\mod v]$ in metric $g'\in \Om_g$, by (\ref{dbd}) we deduce that $$\ta_p<C\ms^v(T).$$ Let $S$ be an area-minimizing integral current in $[\Si].$ Then
		\begin{align*}
			\ms^v(T)\le\ms^v(S)\le\ms (S).
		\end{align*}
		Thus, we have \begin{align}\label{denbd}
			\ta_p<C\ms(S),
		\end{align}
		However, $\ms(S)$, the minimal integral area in $[\Si]$, depends continuously on the metric. Thus, by shrinking the $\Om_g$ if necessary, we can deduce a uniform upper bound. Since $C$ in Lemma \ref{mon} is independent of the metric in $\Om_g$, we are done.  
	\end{proof}
	\begin{lem}\label{dstb}
		With the same assumptions as in Theorem \ref{modvb}, there exists $\ta>0$ so that for any $v\ge 2,$ any mod $v$ area-minimizing current $T$ with $\pd T=\Ga\mod v$, the density of $T$ at any point is smaller than $\ta.$
	\end{lem}
	\begin{proof}
		Let $S$ be an area-minimizing integral current with  boundary $\Ga.$ Then we always have $\ms^v(T)\le\ms^v(S)\le\ms (S).$
		
		Since the measure $V_T$ of $T$ is a stationary integral varifold in $\Ga\cp,$ by compactness of $\Ga$ and (2) Theorem of Section 3.4 in \cite{WA2}, there exists $s>0,c>0$ so that for any $0<r\le s$, and any $p\in\Ga$,
		\begin{align}\label{bdmon}
			\exp(cr)V_T(B_r(p))r^{-d}
		\end{align}
		is monotonically increasing.
		This implies that 
		\begin{align*}
			\ta_p\le \exp(cs)s^{-d}\ms^v(T)\le\exp(cs)s^{-d}\ms(S),
		\end{align*}
		for any point $p\in \Ga.$
		
		On the other hand, for points $p$ at least $\frac{1}{2}s$ away from $\Ga$, by monotonicity formula in \cite{WA1}, we have
		\begin{align*}
			\ta_p\le (\frac{1}{2}s)^{-d}\ms^v(T)=2^ds^{-d}\ms(S).
		\end{align*}
		For $p$ of distance at most $\frac{1}{2}s$ from $\Ga,$ let $q$ be a point minimizing $|p-q|.$ Then
		\begin{align*}
			\ta_p\le& |p-q|^{-d}V_T(B_{|p-q|}(p))\le |p-q|^{-d}V_T(B_{2|p-q|}(q))\\\le& 2^d\exp(c(s-2|p-q|))s^{-d}V(B_s(q))\le2^d\exp(cs)s^{-d}\ms(S)
		\end{align*}
		Combining all the three cases, we are done.
	\end{proof}
	\subsection{Norms on homology}\label{nm}
	In this subsection, we will collect several facts about the minimal mass (Definition \ref{minm}) in homology. First of all the reader needs to know that $\nog{\cdot}$ is non-zero on an integral homology class if and only if the class is non-torsion (\cite{HF2}). Thus, adding torsion classes does not change the norm.
	\begin{lem}\label{norm}
		Let $\no{\cdot}$ be a continuous norm on $\R^n$ and $\{v_1,\cd,v_n\}$ be a basis. Then there exists a $C>0$ that depends continuously on $\no{\cdot},$ so that 
		\begin{align*}
			\sum_j|a_j|\no{v_j}\le C \no{\sum_j a_jv_j},
		\end{align*}with $a_j\in \R.$
	\end{lem}
	\begin{proof}
		For $v=\sum_j a_jv_j$ with $a_j\in\R$ define
		\begin{align*}
			\no{v}^\infty=\sum_j|a_j|\no{v_j}.
		\end{align*}
		It is straightforward to verify that $v$ is a continuous function on $\R^n$ and $v$ is a norm on $\R^n.$ Thus, the lemma is equivalent to $\no{v}^\infty/\no{v}$ is bounded. By homogeneity, it suffices to verify this for $\no{v}=1.$ By continuity of $\no{\cdot}^\infty,$ and the compactness of $\no{v}=1$, we are done.
	\end{proof}
	\begin{lem}\label{msbd}
		With the same assumptions as in Theorem \ref{modvc}, for any fixed metric, there is a constant $D>0,$ so that for any non-torsion homology class $[\Si]\in H_d(M,\Z)$ we have
		\begin{align*}
			\frac{\inf_\Z \ms([\Si])}{\no{\Si}}\le D.
		\end{align*}
	\end{lem}
	\begin{rem}
		For torsion class $[\Si]$, $\no{\Si}=0$ while $\inf_\Z\ms([\Si])>0.$
	\end{rem}
	\begin{proof}
		In the notation of Section \ref{at}, we can write $$\Si=\sum_j\ai_j[\Si_j]+\be_j[\tau_j],$$ with $\ai_j\in \Z,\be_j\in[-N,N]$, where we regard terms with ill-defined subscript as zero. By Lemma \ref{norm} there is $C>0,$ so that
		\begin{align*}
			\no{\Si}=\no{\sum_j[\Si_j]}\ge C\m \sum_j|\ai_j|\no{[\Si_j]}.
		\end{align*}
		By the mediant inequality and non-torsion of $[\Si]$ (i.e., $\max_j|\ai_j|\ge 1$), we have
		\begin{align*}
			&	\frac{\inf_\Z\ms([\Si])}{\no{\Si}}\\\le& C\frac{\sum_j |\ai_j|\inf_\Z\ms([\Si_j])+|\be_j|\inf_\Z\ms([\tau_j])}{\sum_j|\ai_j|\no{\Si_j}}\\
			=&C\frac{\sum_j |\ai_j|\inf_\Z\ms([\Si_j])}{\sum_j|\ai_j|\no{\Si_j}}+C\frac{
				\sum_j|\be_j|\inf_\Z\ms([\tau_j])}{\sum_j|\ai_j|\no{\Si_j}}\\
			\le&C\max_j \frac{\inf_\Z\ms([\Si_j])}{\no{[\Si_j]}}+CNa\frac{\max_j\inf_\Z\ms([\tau_j])}{\min_j\no{\Si_j}},
		\end{align*}where $N,a$ are constants defined in the decomposition of $H_d(M,\Z)$ as in Section \ref{at}.
		The last line is independent of $\ai_j,\be_j,\cd$. We are done.
	\end{proof}
	\subsection{Regularity theorems}
	In this subsection, we will prove some essential lemmas about the regularity of mod $v$ currents.
	\begin{lem}\label{dv2}
		Let $T$ be a mod $v$ area-minimizing current in a not necessarily complete open Riemannian manifold $M$, with $v\ge 2.$ Let $S$ be a rectifiable current that is representative modulo $v$ of $T$ (page 430 of \cite{HF}). Then
		\begin{align*}
			\spt\pd S\s (\textnormal{Reg}(T)\cap\{p|\ta_p<\frac{v}{2}\})\cp,
		\end{align*}where $\textnormal{Reg}(T)$ denotes the regular part of $T.$\end{lem}
	\begin{proof}
		Since $\ta\le \frac{v}{2}$ $d$-dimensional a.e. for $S$ (page 430 of \cite{HF}), it suffices to prove that if $p\in\spt \pd S\cap \textnormal{Reg}(T),$ then $\ta_p= \frac{v}{2}.$ Since $p$ is in the regular set of $T,$ there exists a neighborhood $B_r(p)$ in which $S$ restricted to $B_r(p)$ equals $\ta_p N$ for some smooth submanifold $N$ of $B_r(p)$, so that in some coordinate system $(x_1,\cd,x_{d+c})$, $N$ is the a smooth ball contained in the $x_1\cd x_d$-plane. By Corollary 1.4 of \cite{RY}, there exists a integral flat chain $R$ in $B_r(p)$ so that
		\begin{align*}
			S\res B_r(p)=vR+\ta_p N
		\end{align*} as a rectifiable current. Since $S,N$ are both of finite mass, we deduce that $R$ is also of finite mass, thus a rectifiable current by 4.2.16 of \cite{HF}.
		
		By 4.1.28(5) of \cite{HF}, there exist a $d$-dimensional rectifiable set $B$, so that
		\begin{align*}
			R(\phi)=\int_B\Ta\ri{\eta,\phi} d\hm^d, 
		\end{align*}
		with $\phi$ any smooth $d$-dimensional form, $\eta$ an $\hm^d$ measurable simple unit vector field, $\Ta$ is $\hm^d$ a.e. the positive integer density of $S.$
		
		By page 430 of \cite{HF}, the support of $S$ is contained in $N,$ thus the support of $R$ is also contained in $N,$ i.e., $B\s N.$ This implies that $\eta=\pm \T N$ $\hm^d$-a.e.
		
		Thus, we have
		\begin{align*}
			S\res B_r(p)(\phi)=\int_{N\setminus B}\ta_p \ri{\T_xN,\phi}d\hm^d(x)+\int_{B} (\ta_p\pm v\Ta)\ri{\T_x N,\phi}d\hm^d(x).
		\end{align*} 
		By 4.1.28 of \cite{HF}, we deduce that
		\begin{align*}
			\ms(S\res B_r(p))=\int_{N\setminus B}\ta_pd\hm^d+\int_B |\ta_p\pm v\Ta|d\hm^d.
		\end{align*}
		However, since $\Ta$ is a positive integer on $B$ $\hm^d$-a.e., and $\ta_p\le \frac{v}{2}$, we have 
		\begin{align*}
			|\ta_p\pm v\Ta|\ge \frac{v}{2},
		\end{align*}
		where equality holds if and only if $\ta_p=\frac{v}{2},\Ta=1.$
		
		Thus, we have
		\begin{align*}
			\ms(S\res B_r(p))\ge \ta_p \ms(N)=\ms^v(S\res B_r(p)). 
		\end{align*}
		However, by page 430 of \cite{HF}, $\ms(S\res B_r(p))=\ms^v(S\res B_r(p))$. Since $\Ta>0$ $\hm^d$ a.e., then we must have $\ta_p=\frac{v}{2},$ $\Ta=1$ on $B$ $\hm^d$ a.e.  We are done.
	\end{proof}
	\begin{lem}\label{reg}
		With the same assumption as in Lemma \ref{dst}, assume
		\begin{itemize}
			\item either $\pd T=0\mod v$ or $\pd T=\Ga\mod v,$ where $\Ga=\sum_{j}t_j\Ga_j,$ with $\Ga_j$ a finite collection of disjoint compact connected oriented submanifolds of $M$ and $\frac{v}{2}>t_j\ge 1$ integers,
			\item the density of $T$ at any point $p\in M$ satisfies $\ta_p<\frac{v}{2}$.
		\end{itemize}
		Then $T$ is an integral current with either $\pd T=0$ or $\pd T=\Ga$, respectively.
	\end{lem}
	\begin{proof}
		Let $S$ be a rectifiable current that is a representative modulo $v$ of $T.$ It suffices to prove that $\pd S=\Ga.$
		
		Let us first deal with the case of $\pd T=0\mod v$ using interior density bound. Let us analyze strata by strata using Almgren stratification (\cite{BW}).
		
		By Lemma \ref{dst}, $\pd S$ intersect the regular set of $T$ only at those of density $\ta\ge \frac{v}{2},$ which is empty by our assumptions. Thus, $\pd S$ is supported in the singular set of $T.$
		
		By Theorem 1.6 of \cite{DHMS}
		the flat tangent cone singular points of any mod $v$ area-minimizing current with density less than $\frac{v}{2}$ is of Hausdorff codimension at least $2.$ In the hypersurface case, we can also use \cite{BW1}, which gives the stronger conclusion that no such points exist.
		
		Next, note that points with $d-1$-symmetric tangent cones are empty since every such cone must be $\R^{d-1}$ times $v$ rays originating from the origin, which has density $\frac{v}{2}.$
		
		Since the $d-j$-strata of the singular set has Hausdorff dimension at most $d-j$ (\cite{BW}), we deduce that the singular set of $T$ is of Hausdorff dimension at most $d-2.$ Consequently, $\spt \pd S$ has Hausdorff dimension at most $d-2.$ By 4.1.20 of \cite{HF}, a $d-1$-dimensional flat chain supported in a $d-2$-dimensional integral geometric measure zero set must equal to zero. Since Hausdorff measure zero implies integral geometric measure zero (2.10.6 of \cite{HF}), we deduce that $\pd S=0.$ Thus, $T$ is an integral current.
		
		Now let us deal with the case of $\pd T=\Ga\mod v\not=0\mod v.$ Again, arguing as above, we deduce that $\pd S$ must be supported in $\Ga.$ By 4.1.31 of \cite{HF}, we deduce that
		\begin{align*}
			\pd S=\sum_j s_j\Ga_j,
		\end{align*} with $s_j\in\Z$.
		
		Since $\pd S=\Ga\mod v,$ we have $$s_j-t_j=l_jv$$ with $l_j\in\Z.$ Now consider the Almgren stratification of $T$ at the boundary points $p\in\Ga$, then for $\hm^{d-1}$ a.e., the tangent cone must be $\R^{d-1}$ times $k$ (depending on $p$) number of rays from origin, with $\frac{k}{2}<\frac{v}{2}$. Now if we orient the rays using the orientation of $S,$ then there are $a$ rays pointing away from the origin and $b$ rays pointing towards the origin, with
		\begin{align*}
			a-b=s_j=t_j+l_jv,a+b=k.
		\end{align*}
		This implies that
		\begin{align*}
			0&\le a=\frac{1}{2}t_j+\frac{1}{2}l_jv+\frac{k}{4}<\frac{1}{4}v+\frac{1}{2}l_jv+\frac{v}{4},\\
			0&\le b=-\frac{1}{2}t_j-\frac{1}{2}l_jv+\frac{k}{4}<-\frac{1}{2}l_j v+\frac{v}{4}.
		\end{align*}
		This gives
		\begin{align*}
			-1<l_j<\frac{1}{2}.
		\end{align*}
		Thus, we must have $l_j=0.$ We are done.\end{proof} 
	
	\section{Proof of Theorem \ref{modvc},\ref{modvb},\ref{modva} and Corollary \ref{setk}}
	\subsection{Proof of the theorems}
	For Theorem \ref{modvc}, apply Lemma \ref{dst} and Lemma \ref{reg}. 
	
	For Theorem \ref{modvb} apply Lemma \ref{dstb} and Lemma \ref{reg}.
	
	For Theorem \ref{modva}, by the universal coefficient theorem, a class in $H_d(M,\Z/v\Z)$ admits integral representatives if and only if it lies in the image of tensoring with $\Z/v\Z,$ i.e., taking a class modulo $v.$ 
	
	Write $$H_d(M,\Z)=\Z^{b_d}\oplus G,$$
	where $b_d$ is the $d$-th Betti number of $M$ and $G$ denotes the torsion subgroup of $M.$ Then we have \begin{align*}
		H_d(M,\Z)\ts \Z/v\Z
		=(\Z/v\Z)^{b_d}\oplus (G\ts \Z/v\Z).
	\end{align*}
	Thus, the denominator in our theorem equals $v^{b_d}\times \# G\ts \Z/v\Z.$
	
	For $[\Si]\in G$, since $G$ is finite, by Theorem \ref{modvc}, there exists $\nu>0$ so that for any $v\ge \nu,$ all homology classes in $G\ts \Z/v\Z$ admits only integral minimizers. If $b_d=0$, we are done. Now suppose $b_d>0$ and $v\ge \nu.$ For non-torsion $[\Si]\in H_d(M,\Z)$ that occurs as pre-image of modulo $v$, in the notation of Section \ref{mfd}, write
	\begin{align*}
		[\Si]=\sum_j\ai_j[\Si_j]+\be_j[\tau_j],
	\end{align*}
	with all $\ai_j\in[-\frac{v}{2},\frac{v}{2}]\cap\Z$ and all $\be_j\in[-N,N]\cap\Z.$ Then by Lemma \ref{msbd}, we have
	\begin{align*}
		\inf_{\Z/v\Z}\ms([\Si\mod v])\le\inf_\Z\ms([\Si])\le  D\no{\Si}\le D \sum_j|\ai_j|\no{\Si_j}.
	\end{align*}
	Now apply Lemma \ref{mon} to deduce that for any mod $v$ area-minimizing current in $[\Si]$ and any point $p\in M,$ we have
	\begin{align*}
		\ta_p< C\sum_j|\ai_j|\no{\Si_j}\le Cb_d\sup_j\no{\Si_j}\max_j|\ai_j|.
	\end{align*} 
	Thus, if 
	\begin{align}\label{knbd}
		\max_j|\ai_j|<\frac{v}{2Cb_d\sup_j\no{\Si_j}},
	\end{align} then every area-minimizing mod $v$ current in $[\Si\mod v]$ is integral, by Lemma \ref{dst}. 
	
	Note that asymptotically we have
	\begin{align}\label{asy}
		\#\bigg\{(\ai_1,\cd,\ai_j,\cd)|\max_j|\ai_j|<\frac{v}{2Cb_d\sup_j\no{\Si_j}}\bigg\}=\bigg(\frac{v}{2Cb_d\sup_j\no{\Si_j}}\bigg)^{b_d}+O(v^{b_d-1})
	\end{align}
	
	Note that the right-hand side of the inequality \ref{knbd} is independent of the coefficients $\be_j.$ In other words, the same estimate holds for any other homology class $[\Si']$ so that $[\Si']-[\Si]$ is torsion.  
	Thus, counting all non-torsion $[\Si]$ asymptotically with (\ref{knbd}) amounts  to multiplying (\ref{asy}) with $\# G\ts \Z/v\Z$,  minus the number of torsion classes.
	
	Consequently, breaking down the numerator in the theorem into non-torsion and torsion classes, we have
	\begin{align*}
		&\liminf_v\frac{\#\{ H_d(M,\Z/v\Z)\ni[\Pi]\textnormal{ having only integral area-minimizers}\}}{\#\{ H_d(M,\Z/v\Z)\ni[\Pi]\textnormal{ admitting integral representatives}\}}\\
		\ge&\liminf_v\frac{\Bigg(\bigg(\frac{v}{2Cb_d\sup_j\no{\Si_j}}\bigg)^{b_d}+O(v^{b_d-1})-1\Bigg)\# G\ts\Z/v\Z+\# G\ts\Z/v\Z}{v^{b_d}\times \# G\ts \Z/v\Z}\\
		=&\frac{1+O(v^{-1})}{(2Cb_d\sup_j\no{\Si_j})^{b_d}}>0.
	\end{align*}
	We are done.
	\subsection{Proof of Corollary \ref{setk}}
	First, note that $\inf_{\Z/v\Z}\ms_g([\Si\mod v])$ and $\inf_\Z\ms_g([\Si])$ depend continuously on the metric.
	For the proof, one can use a straightforward adaptation of Lemma 3.7 of \cite{ZLc} with $\Z$ and $\Z/v\Z$ coefficient deformation theorems. We need to work a bit harder for $\inf_{\Z/v\Z}\ms_g([\Si\mod v]).$
	\begin{lem}\label{contm}
		$\inf_{\Z/v\Z}\ms_g([\Si\mod v])$ depends continuously on the metric $g.$
	\end{lem} 
	\begin{proof}
		First, let us prove that in any fixed $g$ the  $\inf_{\Z}\ms_g([\Si\mod v])$ is always achieved on some class $[\Si]+v[\Pi].$ (Note bullet (1) of our corollary follows directly from this.) We argue by contradiction. Suppose not, then there exists a sequence of distinct homology classes $\Pi_l$ so that a sequence of area-minimizing integral current $T_l$ in $[\Si]+v[\Pi_l]$ has strictly decreasing mass. Since the torsion subgroup of $H_d(M,\Z)$ is finite, one can choose a torsion-free subsequence $[\Pi_l]$, not relabeled. In the notation of Section \ref{at}, write $$[\Si]+v[\Pi_l]=\sum_{j_l}\ai_{j_l}[\Si_{j_l}]+\be_{j_l}[\tau_{j_l}].$$ Since $[\Pi_l]$ is an infinite distinct sequence of non-torsion classes, without loss of generality, one can suppose that $\ai_{1_l}\to\infty.$ By Lemma \ref{norm}, this implies that
		\begin{align}\label{infb}
			\ms(T_l)\ge \no{[\Si]+v[\Pi_l]}\ge C\m \sum_{j_l}|\ai_{j_l}|\no{\Si_{j_l}}\ge C\m |\ai_{1_l}|\no{\Si_{1_l}}\to\infty,
		\end{align}a contradiction. 
		
		Next, let us prove that for any metric $g$, there exists an open set $\Om_g$ in the space of Riemannian metrics, so that for any $h\in \Om_g,$ $\inf_{\Z}\ms_h([\Si\mod v])$ is achieved on some fixed finite collection of homology classes, regardless of $h$. By continuous dependence of $\no{\cdot}_h$ on the metric (Lemma 3.7 in \cite{ZLc}) and by Lemma \ref{norm}, there exist an open set $\Om_g\ni g$ so that for any $h\in\Om_g$, we have $\no{\Si}_h\ge C\sum_j |\ai_j|\no{\Si_j}_h$ with some $C>0$ in the notation of decomposition in Section \ref{at}.  Shrink $\Om_g$ if necessary so that $\sup_j\sup_{h\in\Om_g}\no{\Si_j}_h<\infty.$ Now arguing by contradiction, we deduce that there is a sequence of metrics $g_j\to g$ and distinct non-torsion homology classes $[\Pi_j]$ so that $\inf_{\Z}\ms_{g_j}([\Si\mod v])$ is achieved on $[\Si]+v[\Pi_j].$ The same estimate as in (\ref{infb}) in the previous paragraph gives the contradiction, since we have uniform control on $C$ and $\no{\Si_j}_h$ near $g.$
		
		To sum it up, near any fixed metric $g$, there are $k$ different homology classes $[\Pi_l]$ so that
		\begin{align*}
			\inf_{\Z}\ms_h([\Si\mod v])=\min_{j} \inf_{\Z}\ms_h([\Si]+v[\Pi_j]).
		\end{align*} Our lemma then follows from the continuous dependence of $\inf_{\Z}\ms_h([\Si]+v[\Pi_l])$ on the metric and finiteness of $k.$
	\end{proof}
	Now we know that both sides of (\ref{lavi}) depend continuously on the metric. Thus our corollary follows directly from the definition of $K_{[\Si\mod v]}$ and Theorem \ref{modvc}.
	\section{Proof of Theorem \ref{lav} and Theorem \ref{lavb}}
	\subsection{Proof of Theorem \ref{lav}}
	By Lemma \ref{thom}, there is a smooth connected non-orientable submanifold $Q$ of $M$, so that
	\begin{align*}
		[\Si\mod v]=[\frac{v}{2}Q\mod v].
	\end{align*}
	Apply Lemma \ref{retr} with $t=\frac{v}{2},$ we get a smooth metric $g$ on $M,$ so that $\frac{v}{2}Q$ is area-minimizing mod $v$ in $h$. Moreover, there is a neighborhood $U$ that deformation retracts onto $Q$, so that any stationary varifold not contained in $U$ must have an area at least $R\frac{v}{2}\ms(Q).$ Now, let $T$ be an area-minimizing integral current in $[\Si]$. If $T$ is not contained in $U,$ we are done. If $T$ is contained in $U,$ the  $d$-th integral homology of $U$ is the same as the $d$-th integral homology of $Q$, by homotopy invariance. Thus, $H_d(U,\Z)=0$ by non-orientability of $Q$. This is a contradiction, as $T\not=0$ in integral homology. 
	\subsection{Proof of Theorem \ref{lavb}}
	Let $S^{d-1}$ be the unit sphere in $\R^d\s\R^{d+c}.$
	Set
	\begin{align*}
		\Ga_0=S^{d-1}\times \{0\}+S^{d-1}\times \{(\e,0,\cd,0)\}+\cd +S^{d-1}\times \{((v-1)\e,\cd,0)\}.
	\end{align*}
	Note that 
	\begin{align*}
		\Ga_0-vS^{d-1}=\pd\sum_{j=1}^v S^{d-1}\times [0,j\e].
	\end{align*} 
	Now, we can do successive connected sums to make $\Ga_0$ into a connected submanifold $\Ga.$ By making the necks of connected sums with radius at most $\e$, we can make sure that $\Ga$
	\begin{align*}
		\Ga-\Ga_0=\pd C,
	\end{align*} 
	where $C$ has area at most $O(\e^{d}).$
	
	In other words $\Ga-vS^{d-1}=\pd D=\pd(C+\sum_{j=1}^v S^{d-1}\times[0,j\e])$ for some $D$ with area at most $O(\e).$ This implies that any area-minimizing integral current $T$ with $\pd T=\Ga$ must have mass in $[v\ms(S^{d-1})-O(\e),v\ms(S^{d-1})+O(\e)].$ 
	
	On the other hand, 
	\begin{align*}
		\pd (C+\Si_{j=1}^v S^{d-1}\times[0,j\e])\mod v=\Ga-vS^{d-1}\mod v=\Ga.
	\end{align*}
	Thus, any area-minimizing mod $v$ current $T'$ with $\pd T'=\Ga\mod v$ must have an area at most $O(\e).$ Taking $\e$ small, we are done.

\end{document}